\documentclass{amsart}

\usepackage{amsfonts,amsmath,amssymb,amsthm,amscd,amsxtra}
\usepackage{enumerate,verbatim}
\usepackage[all,2cell,ps]{xy}
\usepackage[toc,page]{appendix}

\theoremstyle{plain} 
\newtheorem{theorem}{Theorem}[section]

\newtheorem{proposition}[theorem]{Proposition}
\newtheorem{corollary}[theorem]{Corollary}
\newtheorem{lemma}[theorem]{Lemma}

\newtheorem{conjecture}[theorem]{Conjecture}

\theoremstyle{definition}

\newtheorem{example}[theorem]{Example}
\newtheorem{conj}[theorem]{Conjectures}
\newtheorem*{question}{Question}
\newtheorem{remark}[theorem]{Remark}

\newtheorem{notation}[theorem]{Notation}

\newtheorem{chunk}[theorem]{\hspace*{-1.065ex}\bf}

\numberwithin{equation}{theorem}

\newcommand{\lra}{\longrightarrow}
\newcommand{\hra}{\hookrightarrow}

\newcommand{\fm}{\mathfrak{m}}
\newcommand{\fn}{\mathfrak{n}}
\newcommand{\fp}{\mathfrak{p}}
\newcommand{\fq}{\mathfrak{q}}

\newcommand{\CC}{\mathbb{C}}

\newcommand{\SP}[1]{(S\!P_{#1})}

\DeclareMathOperator{\Ass}{Ass}
\DeclareMathOperator{\ch}{char}
\DeclareMathOperator{\codim}{codim}
\DeclareMathOperator{\depth}{depth}

\DeclareMathOperator{\Ext}{Ext}
\DeclareMathOperator{\height}{height}
\DeclareMathOperator{\Hom}{Hom}
\DeclareMathOperator{\length}{length}
\DeclareMathOperator{\pd}{pd}

\DeclareMathOperator{\Supp}{Supp}
\DeclareMathOperator{\Spec}{Spec}
\DeclareMathOperator{\Tor}{Tor}

\newcommand{\ul}{\underline}

\newcommand{\tensor}{\otimes}

\newcommand{\tf}[2]{{\boldsymbol\bot}_{#1}{#2}}
\newcommand{\tp}[2]{{\boldsymbol\top}_{\hskip-2pt #1}{#2}}

\def\urltilda{\kern -.15em\lower .7ex\hbox{\~{}}\kern .04em}
\def\urldot{\kern -.10em.\kern -.10em}
\def\urlhttp{http\kern -.10em\lower -.1ex
\hbox{:}\kern -.12em\lower 0ex\hbox{/}\kern -.18em\lower 0ex\hbox{/}}

\begin{document}

\title[Vanishing of Tor]
{Criteria for vanishing of Tor \\ over Complete Intersections}
\author[Celikbas, Iyengar, Piepmeyer, Wiegand]
{Olgur Celikbas, Srikanth B.~Iyengar, \\ Greg Piepmeyer, and Roger Wiegand}

\address{Olgur Celikbas \\
University of Connecticut \\
Storrs, CT 06269}
\email{olgur.celikbas@uconn.edu}

\address{Srikanth B.~Iyengar \\
University of Utah\\
Salt Lake City, UT 84112, U.S.A}
\email{iyengar@math.utah.edu}

\address{Greg Piepmeyer \\
Columbia Basin College \\
Pasco, WA 99301, U.S.A}
\email{gpiepmeyer@columbiabasin.edu}

\address{Roger Wiegand\\
University of Nebraska--Lincoln \\
Lincoln, NE 68588, U.S.A}
\email{rwiegand1@math.unl.edu}



\thanks{SBI partly supported by NSF grant DMS-1201889 and a Simons
  Fellowship; RW partly supported by a Simons Collaboration Grant}

\date{16th December 2014}

\begin{abstract} 
In this paper we exploit properties of Dao's  $\eta$-pairing  \cite{Da1} as well as techniques of Huneke, Jorgensen, and Wiegand \cite{HJW} to study the vanishing of $\Tor_i(M, N)$ for finitely generated modules $M$, $N$ over complete intersections.  We prove vanishing of $\Tor_i(M, N)$ for all $i \geq 1$ under depth conditions on $ M $, $ N $, and $ M \tensor N $.  Our arguments improve a result of Dao \cite{Da2} and establish a new
connection between the vanishing of $ \Tor $ and the depth of tensor products.
\end{abstract}

\maketitle{}

\thispagestyle{empty}

\section{Introduction}
In his seminal 1961 paper \cite{Au}, Auslander proved that if $R$ is a local ring and $M$ and $N$
are nonzero finitely generated $R$-modules such that $\pd(M)<\infty$
and $\Tor^{R}_{i}(M,N)=0$ for all $i\geq 1$, then
\begin{equation}\label{eq:DF}
\depth(M)+\depth(N)=\depth(R)+\depth(M\otimes_{R}N)\,,
\end{equation}
 that is, the {\em depth
formula} holds.  
Huneke and Wiegand \cite[Theorem 2.5]{HW1} established the depth
formula for Tor-independent modules (not necessarily of
finite projective dimension) over complete intersection rings. 
Christensen and Jorgensen \cite{CJ} extended that result
 to AB rings \cite{HJ}, a class of Gorenstein rings strictly
containing the class of complete intersections. The depth formula is
important for the study of depths of tensor products of modules
\cite{Au, HW1}, as well as of complexes \cite{Foxby, Iy}. 
 We seek  conditions on the modules $M$, $N$ and
$M\otimes_{R}N$ forcing such a formula to hold, in particular,
conditions implying $\Tor^{R}_{i}(M,N) = 0$ for all $i\geq
1$.  
The following conjecture, implicit in \cite{HJW}, guides our search:

\begin{conjecture}[see \cite{HJW}] 
Let $M$, $N$ be finitely generated modules over a  complete intersection $R$ of codimension $c$. If $M\otimes_{R}N$ is a $(c+1)^\text{st}$ syzygy 
and $M$ has rank,  must $\Tor^{R}_{i}(M,N)=0$ for all $i\geq 1$?
\end{conjecture} 
The conjecture is true if $c=0$ or $1$,  by \cite[Corollary 1]{Li} and \cite[Theorem 2.7]{HW1} respectively.  Without the assumption of rank, there are easy counterexamples, e.g., $R = k[\![x,y]\!]/(xy)$ and $M = N = R/(x)$; $M$ is an $n^\text{th}$ syzygy for all $n$, but the odd index Tors are non-zero.

\medskip

A finitely generated module over a complete intersection is an $n^\text{th}$ syzygy of some finitely generated module if and only if it 
satisfies \emph{Serre's condition} $(S_{n})$; see (\ref{Serre}).  Our methods yield a sharpening of the following theorem due to Dao:

\begin{theorem}[Dao \cite{Da2}]
\label{Long's result} 
Let $R$ be a complete intersection in an unramified regular local ring, of relative codimension $c$, and let $M$, $N$ be
finitely generated $R$-modules.  Assume
\begin{enumerate}[\rm(i)]
\item $M$ and $N$ satisfy $(S_{c})$,
\item $M\otimes_{R}N$ satisfies $(S_{c+1})$, and 
\item $M_{\fp}$ is a free $R_{\fp}$-module for all prime ideals $\fp$
 of height at most $c$.  
\end{enumerate}
Then $\Tor^{R}_{i}(M,N)=0$ for all $i\ge 1$ (and hence the depth formula holds).
\end{theorem}

By analyzing 
Serre's conditions, 
we remove Dao's assumption that the ambient regular local ring
be unramified; see 
Corollary \ref{cor:dao}. Even though complete intersections in unramified
regular local rings suffice for many applications, our conclusion is of
interest: Dao's proof 
uses the nonnegativity of partial Euler characteristics, but
nonnegativity remains unknown for the ramified case; see \cite[Theorem
6.3 and the proof of Lemma 7.7]{Da2}.

If the ambient regular local ring {\em is} unramified, we can replace
 $c$ with $c-1$ in both hypotheses (i) and (ii), 
remove hypothesis (iii),
and still conclude that $\Tor^{R}_{i}(M,N)=0$ for all $i\ge 1$ {\em
  provided} that $\eta_{c}^{R}(M,N)=0$; see (\ref{chunk:eta}) for the
definition of $\eta_c^R(-,-)$ and Theorem \ref{thm:eta2} for our
result.

Moore, Piepmeyer, Spiroff, and Walker \cite{MPSW2},\cite{Walker} have proved vanishing
of the $\eta$-pairing in several important cases.  These, in turn, yield results on vanishing of Tor.  See 
Proposition~\ref{prop:MPSW-local},  Theorem~\ref{thm:Walker}, and Corollary~\ref{cor:Walker-local}.

Our proofs rely on a reduction technique using {\em quasi-liftings}; see
(\ref{pushforward}).  Quasi-liftings were initially defined and
studied by Huneke, Jorgensen and Wiegand in \cite{HJW}. Lemma
\ref{prop:eta} is the key ingredient for our argument. It shows that
if $R=S/(f)$ and $S$ is a complete intersection of codimension $c-1$,
and if $\eta^{R}_{c}(M,N)=0$, then $\eta^{S}_{c-1}(E,F)=0$, where $E$
and $F$ are quasi-liftings of $M$ and $N$ to $S$, respectively. By
induction, we obtain that $\Tor_{i}^S(E,F)=0$ for all $i\geq 1$: this
allows us to prove the vanishing of $\Tor_{i}^R(M, N)$ from the depth
and syzygy relations between the pairs $E, F$ and $M,N$.

In the Appendix we revisit the paper of Huneke and Wiegand \cite{HW1} 
 and use our work 
to obtain one of the main results there.  Moreover, we point out
an oversight in Miller's paper \cite{Mi} and state her result in its corrected
form as Corollary \ref{cor:hyp-power}.

\section{Preliminaries}
\label{sec:prelims}
We review a few concepts and results, especially universal pushforwards and quasi-liftings \cite{HJW,HW1}.
Throughout $R$ will be a commutative noetherian ring.

Let $\nu_{R}(M)$ denote the minimal number of generators of the $R$-module $M$. If $(R,\fm)$ is local, the  \emph{codimension} of $R$ is $\codim(R):=\nu_{R}(\fm) - \dim(R)$; it is a nonnegative integer.  We have ${\codim(\widehat{R})} = \codim(R)$, where $\widehat{R}$ is the $\fm$-adic completion of $R$.

\begin{chunk} \textbf{Complete intersections.} 
\label{sec:ci} 
$R$ is a \emph{complete intersection in a local ring} $(Q,\fn)$ if there a surjection $\pi\colon Q\twoheadrightarrow R$  with $\ker(\pi)$ generated by a $Q$-regular sequence in $\fn$; the length of this regular sequence is the \emph{relative codimension of $R$ in $Q$}. A \emph{hypersurface in} $Q$ is a complete intersection of relative codimension one in $Q$.

Assume $\smash{\widehat{R}}$ is a complete intersection in a regular local ring $(Q,\fn)$, of relative codimension $c$. Then $\smash{\widehat{R}} = Q/(\underline f)$ for a regular sequence $\underline f=f_1,\dots,f_c$, where $\codim(R) \le c$. Moreover, the codimension of $ R $ is $c$ if and only if $(\underline f)\subseteq \fn^2$.

A ring is a \emph{complete intersection} (resp., \emph{hypersurface}) if it is local and its completion is a complete intersection (resp.,~hypersurface) in a regular local ring.
\end{chunk}

\begin{chunk} 
\textbf{Ramified regular local  rings.} 
\label{subsec:ram} 
A regular local ring $(Q,\fn, k)$ is said to be \emph{unramified} if either (i) $Q$ is equicharacteristic, i.e., contains a field, or else (ii) $Q \supset  \mathbb Z$, $\ch(k)=p$, and $p \notin \fn^2$.  In contrast, the regular local ring $R=V[x]/(x^2-p)$, where $V$ is the ring of $p$-adic integers, is \emph{ramified}.  Every localization, at a prime ideal, of  an unramified regular local ring is again unramified; see  \cite[Lemma 3.4]{Au}.

  Let $(Q,\fn, k)$ be a $d$-dimensional complete regular local ring.
  If $Q$ is ramified, then $k$ has characteristic $ p $.  Further,
  there is a complete unramified discrete valuation ring $(V, pV)$
  such that $Q\cong T/(p-f)$, where $T=V[[x_1, \dots, x_{d}]]$ and $f$
  is contained in the square of the maximal ideal of $T$; see for
  example \cite[Chaper IX, \S3]{Bourbaki}. Hence every complete
  regular local ring is a hypersurface in an unramified one.
  Consequently, when $R$ is a complete intersection, $\widehat R$ is
  a complete intersection in an unramified regular local ring $Q$
  such that $\codim R \leq c \leq \codim R+1$, where $c$ is the
  relative codimension of $\widehat R$ in $Q$.
\end{chunk}

\begin{chunk}
\textbf{The depth formula} \label{dfor} %
(\cite[Theorem 2.5]{HW1})\textbf{.} Let $R$ be a complete intersection and let $M$, $N$ be finitely generated $R$-modules. If $\Tor_{i}^{R}(M,N)=0$ for all $i\geq 1$, then the \emph{depth formula} \eqref{eq:DF} holds, that is,
\[
\depth(M)+\depth(N)=\depth(R)+\depth(M\otimes_{R}N)\,.
\]
Recall that $\depth(0)=\infty$, so the formula  holds trivially if  a zero module appears.  
\end{chunk}

\begin{chunk}
  \textbf{Torsion submodule.} \label{ts} %
The \emph{torsion submodule} $\tp R M$ of $M$ is the kernel of the natural homomorphism $M\to \text{Q}(R)\otimes_RM$, where $\text{Q}(R)  = \{\text{non-zerodivisors}\}^{-1}R$ is the total quotient ring of $R$. The module $M$ is \emph{torsion} if $\tp RM=M$, and \emph{torsion-free} if $\tp RM = 0$.  To restate, $M$ is torsion-free if and only if every non-zerodivisor of $R$ is a non-zerodivisor on $M$, that is, if and only if $\bigcup\Ass M \subseteq \bigcup\Ass R$. Similarly, $M$ is torsion if and only if $M_\fp = 0$ for all $\fp\in \Ass(R)$.  For notation, the inclusion $\tp RM\subseteq M$ has cokernel $ \tf RM $:
\begin{equation}
\label{eq:tp}
0\lra \tp RM \lra M\lra \tf RM\lra 0\,.
\end{equation}
\end{chunk}

\begin{chunk} \textbf{Torsionless and reflexive
    modules.} \label{TRM} %
Let $M$ be a finitely generated $R$-module; $M^{*}$ denotes its dual $\Hom_R(M,R)$. The module $M$ is \emph{torsionless} if it embeds in a free module, equivalently, the canonical map $M\to M^{**}$ is injective.  Torsionless modules are torsion-free, and the converse holds if $R_\fp$ is Gorenstein for every associated prime $\fp$ of $R$;  see \cite[Theorem A.1]{Vas}.  The module $M$ is \emph{reflexive} provided the map $M\to M^{**}$ is an isomorphism.
\end{chunk}

\begin{chunk} \label{Serre} 
\textbf{Serre's conditions} 
(see \cite[Appendix A, \S1]{LW} and \cite[Theorem 3.8]{EG})\textbf{.}  Let
  $M$ be a finitely generated $R$-module and let $n$ be a nonnegative
  integer. Then $M$ is said to satisfy \emph{Serre's condition}
  $(S_n)$ provided that
\begin{equation}\notag{}
  \depth_{R_{\fp}}(M_\fp)\ge \min\{n,\height(\fp)\} \; \text{for all
    $\fp\in\Supp(M)$}. 
\end{equation}

A finitely generated module $M$ over a local ring $R$ is \emph{maximal
  Cohen-Macaulay} if $\depth(M) = \dim(R)$; necessary for this
equality is that $ M \neq 0 $.  %

If $M$ satisfies $(S_1)$, then $M$ is torsion-free, and the converse
holds if $R$ has no embedded primes, e.g., is reduced or
Cohen-Macaulay; see (\ref{ts}).
If $R$ is Gorenstein, $M$ satisfies $(S_{2})$ if and only if $M$ is
reflexive; see (\ref{TRM}) and \cite[Theorem 3.6]{EG}.  Moreover, if
$R$ is Gorenstein, $M$ satisfies $(S_n)$ if and only if $M$ is an
$n^{\text{th}}$ syzygy module; see \cite[Corollary A.12]{LW}.
\end{chunk}

A localization of a torsion-free module need not be torsion-free; see,
for example, \cite[Example 3.9]{NY}. However, over Cohen-Macaulay
rings, we have:

\begin{remark} 
\label{lem:vb-torsionfree} 
Assume $R$ is Cohen-Macaulay and $M$ is a finitely generated $R$-module. Let $\fp$ be a prime ideal of $R$. Note that, since $\tp RM$ is killed by a non-zerodivisor of $R$, $(\tp RM)_\fp$ is a torsion $R_\fp$-module. Next, $ \tf RM $ satisfies $ (S_1) $ as $R$  is Cohen-Macaulay, and so $(\tf RM)_\fp$ is a torsion-free $R_{\fp}$-module; see (\ref{Serre}).  Localizing the exact sequence (\ref{eq:tp}) at $\fp$, we see that $(\tp R M)_{\fp} \cong \tp {R_\fp}({M_{\fp})}$.  In particular, if $M$ is a torsion-free $R$-module, then $M_{\fp}$ is a torsion-free $R_{\fp}$-module.  
\end{remark}

We recall a technique from \cite[\S1]{HJW} for lowering the codimension. 

\begin{chunk} \textbf{Pushforward and quasi-lifting} (see \cite[\S1]{HJW})\textbf{.}
 \label{pushforward} 
  Let $R$ be a Gorenstein
  local ring and let $M$ be a finitely generated torsion-free
  $R$-module. Choose a surjection $\varepsilon\colon
  R^{(\nu)}\twoheadrightarrow M^*$ with $\nu= \nu_R(M^*)$. Applying
  $\Hom(-,R)$ to this surjection, we obtain an injection
  $\varepsilon^*\colon M^{**} \hra R^{(\nu)}$. Let $M_1$ be the
  cokernel of the composition $M\hra M^{**} \hra R^{(\nu)}$. The exact
  sequence
\begin{equation}
\label{eq:pushforward}
0 \to M\to R^{(\nu)} \to M_1 \to 0
\end{equation}
is called a \emph{pushforward} of $M$.  The extension
\eqref{eq:pushforward} and the module $M_1$ are unique up to
non-canonical isomorphism; see \cite[pp.~174--175]{Ce}. We
refer to such a module $M_1$ as the pushforward of $M$.  Note 
$M_{1}=0$ if and only if $M$ is free.

Assume $R=S/(f)$ where $(S,\mathfrak{n})$ is a local ring and $f$ is a non-zerodivisor in $\mathfrak{n}$.  Let $S^{(\nu)}\twoheadrightarrow M_{1}$ be the composition of the canonical map $S^{(\nu)}\twoheadrightarrow R^{(\nu)}$ and the map $R^{(\nu)}\twoheadrightarrow M_{1}$ in \eqref{eq:pushforward}. The \emph{quasi-lifting} of $M$ to $S$ is the module $E$ in the exact sequence of $S$-modules:
\begin{equation}
\label{eq:quasilifting}
    0 \to E \to S^{(\nu)} \to M_{1} \to 0\, . 
\end{equation}
The quasi-lifting of $M$ is unique up to isomorphism of $S$-modules.
\end{chunk}

Proposition~\ref{prop:push-lift} is from 
\cite[Propositions~1.6 \& 1.7]{HJW}; Proposition~\ref{prop:push-lift2} is
embedded in the proofs of 
\cite[Propositions 1.8 \& 2.4]{HJW} and is recorded explicitly in \cite[Proposition
3.2(3)(b)]{Ce}. We will use Proposition~\ref{prop:push-lift2} in the proofs of
Theorem~\ref{thm:eta2} and Theorem~\ref{prop:second-rigidity} below.

\begin{proposition}[\cite{HJW}]\label{prop:push-lift}
  Let $R$ be a Gorenstein local ring and let $M$ be a finitely
  generated torsion-free $R$-module.  Let $M_{1}$ denote the
  pushforward of $M$.
\begin{enumerate}[\rm(i)]
\item Let $n\ge 0$. Then $M$ satisfies $(S_{n+1})$ if and only if
  $M_1$ satisfies $(S_n)$.
\item Let $\fp$ be a prime ideal. 
  If $M_{\fp}$ is a maximal Cohen-Macaulay $R_{\fp}$-module, then
  $(M_{1})_{\fp}$ is either zero or a maximal Cohen-Macaulay
  $R_{\fp}$-module.
\end{enumerate}
\end{proposition}

\begin{proposition}  [\cite
{HJW}]\label{prop:push-lift2} 
Let $R=S/(f)$ where $S$ is a  complete intersection and $f$ is a
  non-zerodivisor in $S$. Let $N$ be a finitely generated
  torsion-free $R$-module such that $M\otimes_{R}N$ is
  reflexive. Assume $\Tor^{R}_{i}(M,N)_{\fp}=0$ for all $i\geq 1$, and
  for all primes $\fp$ of $R$ with $\height(\fp)\leq 1$.
\begin{enumerate}[\rm(i)] 
\item Then $M_{1}\otimes_{R}N$ is torsion-free. 
\item Let $E$ and $F$ denote the quasi-liftings of $M$ and $N$ to $S$, respectively; see
(\ref{pushforward}).  Assume $\Tor^{S}_{i}(E,F)=0$ for all $i\geq 1$.  
Then $\Tor^{R}_{i}(M,N)=0$
for all $i\geq 1$.
\end{enumerate} 
\end{proposition}

Serre's conditions $(S_{n})$ need not ascend along flat local homomorphisms.  This can be problematic:

\begin{example}
\label{ex:Emmy}
The ring $\mathbb C[[x,y,u,v]]/(x^2,xy)$ has depth two and therefore, by Heitmann's theorem \cite[Theorem 8]{He}, it is the completion $\widehat R$ of a unique factorization domain $(R,\fm)$.  Then $R$, being normal, satisfies $(S_2)$, but $\widehat R$ does not even satisfy $(S_1)$, since the localization at the height-one prime ideal $(x,y)$ has depth zero.
\end{example}

For flat local homomorphisms between Cohen-Macaulay rings, and more generally when the fibers are Cohen-Macaulay, however, $(S_n)$ \emph{does} ascend and descend:

\begin{lemma}
\label{lem:ascent}
Let $R$ be a local ring, $\fp$ a prime ideal of $R$, and $M$  a finitely generated $R$-module.
\begin{enumerate}[{\quad\rm(1)}]
\item If $M$ is reflexive, then so is the $R_\fp$-module $M_\fp$.
\item Suppose $R$ is Cohen-Macaulay.  Then
$(\tp R M)_{\fp} = \tp {R_\fp}{M_{\fp}}$; in particular, if $M$ is torsion-free, 
then so is $M_{\fp}$.
\item Suppose $R\to S$ is a flat local homomorphism.  If $S\otimes_RM$ satisfies $(S_n)$ as an $S$-module, then $M$ satisfies  $(S_n)$ as an $R$-module; the converse holds when the fibers of the map $R\to S$ are Cohen-Macaulay.
\end{enumerate}
\end{lemma}

\begin{proof}
For part (1), localize the isomorphism $M\to M^{**}$.  Part  (2) is Remark~\ref{lem:vb-torsionfree}.  Part (3) can be proved along the same lines as \cite[Theorem 23.9]{Mat}: For any $\fq$ in $\Spec S$ with $\fp=\fq\cap R$, it follows from \cite[Theorem 15.1 and Theorem 23.3]{Mat} that
\begin{align*}
\height(\fq) &= \height(\fp) + \dim (S_{\fq}/\fp S_{\fq}) \quad\text{and}\\
\depth_{S_{\fq}}(S\otimes_{R}M)_{\fq} &= \depth_{R_{\fp}}(M_{\fp}) + \depth (S_{\fq}/\fp S_{\fq})\,.
\end{align*}
When $S\otimes_RM$ satisfies $(S_n)$, for $\fq$ minimal in $S/\fp S$ these equalities give
\[
\depth_{R_{\fp}}(M_{\fp}) = \depth_{S_{\fq}}(S\otimes_{R}M)_{\fq} \geq \min\{n,\height(\fq)\} = \min\{n,\height(\fp)\}.
\]
Thus $M$ satisfies $(S_{n})$. Conversely, if $S_{\fq}/\fp S_{\fq}$ is Cohen-Macaulay and the $R$-module $M$ satisfies $(S_{n})$, one gets
\[
\depth_{S_{\fq}}(S\otimes_{R}M)_{\fq} \geq \min\{n,\height(\fp)\} + \dim (S_{\fq}/\fp S_{\fq}) \geq \min\{n,\height(\fq)\}.
\]
This completes the proof of part (3).
\end{proof}

\section{Main theorem} \label{sec:eta}

Our main result, Theorem \ref{thm:eta2}, is here. We use the $\theta$ and $\eta$-pairings introduced by Hochster~\cite{Ho} and
Dao~\cite{Da2}. After  preliminaries on these, we focus on complete intersections; see (\ref{sec:ci}), the setting of our applications.  

\begin{chunk}{\bf The $\theta$ and $\eta$ pairings}  (Hochster \cite{Ho} and Dao \cite{Da1, Da2})\textbf{.}
\label{chunk:eta} 
Let $R$ be a local ring and let $M$ and $N$ be finitely generated $R$-modules. Assume that there exists an integer $f$ (depending on $M$ and $N$), such that $\Tor_i^R(M,N)$ has finite length for all $i \ge f$.

If $R$ is a hypersurface, then $\Tor^{R}_{i}(M,N)\cong \Tor^{R}_{i+2}(M, N)$ for all $i\gg 0$; see \cite{Ei}. Hochster \cite{Ho} introduced the $\theta$ pairing as follows:
\begin{equation*}
\theta_{}^R(M,N) = \length( \Tor^{R}_{2n}(M,N)) - \length(\Tor^{R}_{2n-1}(M, N)) \text{ for } n\gg 0\, .
\end{equation*}

When $R$ is any complete intersection, Dao \cite[Definition 4.2.]{Da2} defined:
\begin{equation*}
\eta_e^R(M,N) = \lim_{n\to\infty}
\frac{1}{n^e}\sum_{i=f}\limits^{n}(-1)^{i} \length(
\Tor^{R}_{i}(M,N))\, .
\end{equation*}
The $\eta$-pairing is a natural extension to complete intersections of
the $\theta$-pairing. Moreover the following statements hold; see
\cite[4.3]{Da2}.

\begin{enumerate}[{\rm(i)}]
\item $\eta_e^R(M,-)$ and $ \eta_e^R(-,N) $ are additive on short
  exact sequences, provided $ \eta_e^R $ is defined on the pairs of modules
involved.  
\item If $R$ is a hypersurface, then $
  \eta_1^R(M,N)= \frac{1}{2}\theta^R(M,N)$. Hence $\eta^R_1(M,N)=0$ if and only if
  $\theta^R(M,N)=0$.
\end{enumerate}
Assume $R$ is a complete intersection.
\begin{enumerate}[{\rm(i)}]\setcounter{enumi}{2}
\item $\eta_{e}^{R}(M,N)=0$ if $e\ge\codim R$ and either $M$ or $N$
  has finite length. 
\item $\eta_e^R$ is finite when $ e = \codim(R)  $, and $ \eta_e^R $
  is zero when $e > \codim R$. 
\end{enumerate}
\end{chunk}

The next result (Dao \cite[Theorem 6.3]{Da2}), on {\em Tor-rigidity},  
shows the utility of the $\eta$-pairing.

\begin{theorem}[Dao \cite{Da2}] 
  \label{prop:c-rigid} 
Let $R$ be a local ring whose completion is a complete intersection, of relative codimension $c\ge1$, in an \emph{unramified} regular local ring.  Let $M, N$ be finitely generated $R$-modules. Assume $\Tor^R_i(M,N)$ has finite length for all $i\gg0$, and that $\eta^R_c(M,N) = 0$. Then the pair
$M,N$ is \ $c$-Tor-rigid, that is, if $s\ge0$ and $\Tor_i^R(M,N)=0$ for all $i =s,\dots,s+c-1$, then $\Tor_i^R(M,N) = 0$ for all $i\ge s$.
\end{theorem}

The following conjectures have received quite a bit of attention:

\begin{conj} \label{conjD} %
  Assume $R$ is a local ring which is an
  isolated singularity, i.e., $R_{\fp}$ is a regular local ring for
  all non-maximal prime ideals $\fp$ of $R$.
\begin{enumerate}[\rm(i)]
\item (Dao \cite[Conjecture 3.15]{Da1}) %
  If $R$ is an equicharacteristic hypersurface of even dimension, then
  $\eta^{R}_{1}(M,N)=0$ for all finitely generated $R$-modules $M,N$.
\item (Moore, Piepmeyer, Spiroff and Walker \cite[Conjecture
  2.4]{MPSW2}) %
  If $R$ is a complete intersection of codimension $c\ge 2$, then
  $\eta^R_{c}(M,N)=0$ for all finitely generated $R$-modules $M,N$.
\end{enumerate}
\end{conj}

Moore, Piepmeyer, Spiroff and Walker~\cite{MPSW} have settled
Conjecture \ref{conjD}(i) in the affirmative for certain types of
affine algebras. Polishchuk and Vaintrob \cite[Remark 4.1.5]{PV}, as
well as Buchweitz and Van Straten \cite[Main Theorem]{BS}, have since
given other proofs, in somewhat different contexts, of this result;
see Theorem \ref{thm:Walker} for a recent result of Walker
\cite{Walker} concerning Conjecture \ref{conjD}(ii), and Corollary
\ref{cor:Walker-local} for an application of his result.

Our proofs of Lemma \ref{thm:eta} and Theorem
\ref{prop:second-rigidity} use the following  (see \cite[Lemma 3.1]{Au} or \cite[Lemma
  1.1]{HW1}).

\begin{remark} 
\label{Noproof} 
Let $R$ be a  local ring, and let $M$ and $N$ be nonzero finitely generated $R$-modules. Assume $M\otimes_{R}N$ is torsion-free. Then $M\otimes_{R}N \cong M \otimes \tf RN$. Moreover, if $\Tor_1^R(M,\tf RN) = 0$,  then $\tp R N=0$, and hence $N$ is torsion-free.
\end{remark}

We encounter the same hypotheses often enough to warrant a piece of notation.

\begin{notation}
\label{not:SP}
Let $c$ be a positive integer. A pair $M, N$ of finitely generated
modules over a ring $R$ satisfies $\SP c$ provided the following
conditions hold:
\begin{enumerate}[{\rm(i)}]
\item \label{item:M-N-Sc} %
$M$ and $N$ satisfy Serre's condition $(S_{c-1})$.
\item \label{item:MoN-Sc} %
$M\otimes_{R}N$ satisfies $(S_{c})$.
\item \label{item:tor-fin-len} %
$\Tor^R_i(M,N)$ has finite length for all $i\gg0$.
\end{enumerate}
\end{notation}

\subsection*{\em Hypersurfaces} 
We begin with a lemma analogous to \cite[Proposition 3.1]{Da3};
however, we do not assume any depth properties on $M$ or $N$; see (\ref{sec:ci})
and (\ref{not:SP}).

\begin{lemma} 
\label{thm:eta} 
Let $R$ be a local ring whose completion is a hypersurface in an \emph{unramified} regular local ring, and let $M, N$ be finitely generated $R$-modules. Assume the following hold:
\begin{enumerate}[\rm(i)]
\item $\dim(R)\ge 1$. 
\item The pair $M,N$ satisfies $\SP 1$. 
\item $\Supp_{R}(\tp RN)\subseteq \Supp_{R}(M)$.
\item $\theta^R(M,N)=0$.
\end{enumerate}
Then $\Tor^{R}_{i}(M,N)=0$ for all $i\ge 1$, and $N$ is torsion-free.
\end{lemma}
\begin{proof} 

Consider the following conditions for a prime ideal $\fp$ of $R$:
\begin{equation}\tag{\ref{thm:eta}.1}
(\tp RN)_\fp \text{ has finite length over } R_{\fp}, \text{ and } \dim(R_\fp) \ge 1.
\end{equation}

\noindent
{\sc Claim:} If $\fp$ is 
as in (\ref{thm:eta}.1), then $\Tor^{R_{\fp}}_{i}(M_{\fp}, (\tf
RN)_\fp)=0$ for all $i\geq 1$.  \smallskip

We may assume that
$M_{\fp}\ne 0$. We know from (ii) that
$\Tor^{R_{\fp}}_{i}(M_{\fp},N_{\fp})$ has finite length over $R_{\fp}$
for all $i\gg 0$. Since $(\tp RN)_\fp$ has finite length,
the exact sequence \eqref{eq:tp} for $N$, localized at $\fp$, shows that
$\Tor^{R_{\fp}}_{i}(M_{\fp},(\tf RN)_\fp)$ has finite length over
$R_{\fp}$ for all $i\gg 0$.

Using the additivity of $\theta^{R_\fp}$ along the same exact sequence, we see that
\begin{equation}
\label{eq:add}
\theta^{R_\fp}(M_\fp,(\tf RN)_\fp) = - \theta^{R_\fp}(M_\fp,(\tp RN)_\fp) =0\,,
\end{equation}
the last by (\ref{chunk:eta}).

Since $\tf RN$ is a torsionless $R$-module (see (\ref{TRM})), there exists an  exact sequence 
\begin{equation} 
\tag{\ref{thm:eta}.2}
0\to \tf RN \to R^{(n)}\to Z \to 0\,.
\end{equation}
Localizing this sequence at $\fp$, we see that, for $i\gg0$,  $\Tor^{R_{\fp}}_{i}(M_{\fp},Z_\fp)$ has finite length and hence (since $\dim(R_\fp) \ge 1$) is torsion. Now Corollary~\ref{cor:tortor} forces $\Tor^{R_{\fp}}_{i}(M_{\fp}, Z_\fp)$  to be torsion for all $i\geq 1$.

From (\ref{thm:eta}.2), we see that $\Tor^{R_{\fp}}_{1}(M_{\fp}, Z_\fp)$ embeds into $M_{\fp}\otimes_{R_{\fp}}(\tf RN)_\fp$.  But 
$\Tor^{R_{\fp}}_{1}(M_{\fp},Z_\fp)$ is torsion, and (by Remarks \ref{lem:vb-torsionfree} and \ref{Noproof})  $ M_{\fp}\otimes_{R_{\fp}}(\tf RN)_\fp$ is
torsion-free; therefore $\Tor^{R_{\fp}}_{1}(M_{\fp}, Z_\fp)=0$.

Next we note that $\theta^{R_\fp}(M_\fp,
Z_\fp)= - \theta^{R_\fp}(M_\fp,(\tf RN)_\fp) =0$; see (\ref{thm:eta}.2)
and \eqref{eq:add}.
This implies, by Theorem \ref{prop:c-rigid}, that
$\Tor^{R_{\fp}}_{i}(M_{\fp}, Z_\fp)=0$ for all $i\geq 1$; see
(\ref{chunk:eta}). The  claim now follows from (\ref{thm:eta}.2).

If $\tp RN \ne 0$, then there is a prime $\fp$ minimal in $\Supp_R(\tp
RN)$, and so $(\tp RN)_\fp$ is a nonzero  module of finite length. 
Moreover $\dim(R_\fp)\geq 1$: otherwise $\fp\in \Ass(R)$ and
hence $(\tp RN)_{\fp} = 0$; see (\ref{ts}). Thus $\fp$ satisfies
(\ref{thm:eta}.1) and, by our claim, $\Tor^{R_{\fp}}_{i}(M_{\fp}, (\tf
RN)_\fp)=0$ for $i\ge1$. The hypothesis (iii) on supports implies that $M_{\fp} \neq
0$, and now  Remark~\ref{Noproof} yields a
contradiction. We conclude that $\tp RN= 0$.

Applying the claim to the maximal ideal $\fp$ of $R$ yields the required vanishing.
\end{proof}

\begin{remark} 
\label{qu:eta2} 
\label{eq:node} 
\label{rmks:Supp} $\phantom{}$
\begin{enumerate}[\rm(i)]
\item The hypothesis (iii) of Lemma \ref{thm:eta} holds when, for
  example, the support of $N$ is contained in that of $M$. Moreover,
  if $R$ is a domain and $M$ and $N$ are nonzero, then, since $M\otimes_{R}N$
  is torsion-free, we see that
  $\Supp(M\otimes_RN)=\Spec(R)$, whence $\Supp(M) = \Spec(R)$.

\item Most of the hypotheses in Lemma \ref{thm:eta} are essential; see
  the discussion after \cite[Remark 1.5]{HW2}. Notice, without the
  assumption that $\dim(R)\ge 1$, the lemma would fail. Take, for
  example, $R=\CC[x]/(x^2)$ and $M = R/(x) = N$. The vanishing of
  $\theta$ is also essential: let $R = \CC[[x,y]]/(xy)$, $M = R/(x)$
  and $N = R/(x^2)$. Then the pair $M,N$ satisfies the conditions
  (ii) and (iii) of Lemma \ref{thm:eta}. On the other hand
  $\Tor^{R}_{2i+1}(M,N)\cong k$ for all $i\ge 0$, and
  $\Tor^{R}_{2i}(M,N)=0$ for all $i\geq 1$. (Thus $\theta^R(M,N) = -1$.)
\end{enumerate}
\end{remark}

The completion of any regular ring is a hypersurface in an unramified
regular local ring; see (\ref{subsec:ram}). Hence the following
consequence of Lemma \ref{thm:eta} extends Lichtenbaum's 
\cite[Corollary 3]{Li}, which in turn builds on Auslander's
\cite[Theorem~3.2]{Au}; cf. C. Miller's result recorded as
Corollary~\ref{cor:hyp-power} here.

\begin{proposition} \label{prop:powers} %
  Let $(R, \fm)$ be a $d$-dimensional local ring whose completion is a
  hypersurface in an unramified regular local ring, with $d\geq 1$,
  and let $M$ be a finitely generated $R$-module. Assume
  $\pd_{R_{\fp}}(M_{\fp})<\infty$ for all prime ideals $\fp\neq \fm$
  and that 
  $\theta^{R}(M,-)=0$. If $\tensor_R^nM$ is torsion-free for some
  integer $n\ge 2$, then $\pd(M) \leq (d - 1)/n$.  Consequently, if
  $M$ is not free, then $\tensor_R^nM$ has torsion for each $n\ge
  \max\{2,d\}$.
\end{proposition}

\begin{proof}
  We may assume $M\ne 0$.  
Iterating 
  Lemma~\ref{thm:eta} shows that $\tensor_R^pM$ is torsion-free for $p=1,\dots, n$, and that
  $\Tor^{R}_{i}(M,\tensor_R^{p-1}M)=0$ for all $i\geq 1$. Taking $p=2$, we
  see from \cite[Theorem 1.9]{HW2} that $\pd(M)<\infty$. Since
  $\depth({\tensor_R^nM})\geq 1$, one obtains, using \cite[Corollary 1.3]{Au} and the Auslander-Buchsbaum formula \cite[Theorem 3.7]{AuBu}, $n \cdot \pd(M) =
  \pd({\tensor_R^nM}) = d - \depth({\tensor_R^nM}) \leq d - 1$.
\end{proof}

\subsection*{\em Complete intersections}
Hypersurfaces in complete intersections give the inductive step for
our proof of Theorem \ref{thm:eta2}; see (\ref{pushforward}) on pushforwards.

\begin{lemma} \label{prop:eta} %
  Let $(S,\fn)$ be a complete intersection, and let $R$ be a
  hypersurface in $S$. Let $M$ and $N$ be finitely generated
  torsion-free $R$-modules, and let $E$ and $F$ be the quasi-liftings
  of $M$ and $N$, respectively, to $S$.  Assume $\Tor^{R}_i(M,N)$ has
  finite length for all $i \gg 0$.  
  Let $e$ be an integer with  $e\geq \max\{2,\codim(S)+1\}$.
   Then
\begin{enumerate}[\rm(i)]
\item $\Tor^S_i(E,F)$ has finite length for all $i\gg 0$, and 
\item $\eta_{e-1}^{S}(E,F) = 2\cdot e \cdot \eta_{e}^{R}(M,N)$. 
\end{enumerate}
\end{lemma}

\begin{proof}
  By hypothesis, $R\cong S/(f)$, where $f$ is a
  non-zerodivisor in $S$. The spectral sequence associated to the
  change of rings $S\to R$ yields the following exact sequence, see
  \cite[pp. 223--224]{Li} or \cite[p. 561]{Mu}, for all $n\geq 1$:
\[
\cdots\to \Tor^R_{n-1}(M,N) \to \Tor_n^S(M,N) \to \Tor_n^R(M,N)\to
\cdots 
\]
Consequently $\Tor^{S}_{i}(M,N)$ has finite length for $i\gg 0$. Let
$M_1$ and $N_1$ be the pushforwards of $M$ and $N$, respectively.
Since $\Tor^{S}_{i}(R,-)=0$ for all $i\geq 2$, the sequences  
\eqref{eq:quasilifting} and \eqref{eq:pushforward} yield isomorphisms
\[
\Tor^{S}_{i}(E,N) \cong \Tor^{S}_{i+1}(M_1,N) \cong \Tor^{S}_{i}(M,N)
\; \text{for all $i\geq 2$}\, .  
\]
Arguing in the same vein, one gets isomorphisms
\[
\Tor^{S}_{i}(E,F)\cong \Tor^{S}_{i}(E,N) \; \text{for all $i\geq 2$}.
\]
Hence the length of $\Tor^{S}_{i}(E,F)$ is finite for all $i\gg 0$, and
so (i) holds.

Similar arguments show the $\eta$-pairing, over both $ R $ and $ S
$, as appropriate, is defined for all pairs $(X,Y)$ with
$X\in\{M,M_{1},E\}$ and $Y\in\{N, N_{1},F\}$.

By hypothesis, $\codim(S)\leq e-1$, and hence $\codim(R)\leq e$;
see
(\ref{sec:ci}). 
Additivity of $\eta$ along the exact sequences \eqref{eq:pushforward}
and \eqref{eq:quasilifting} thus gives 
\begin{equation*}
\begin{split}
\eta_{e}^{R}(M,N) & = - \eta_{e}^{R}(M_{1},N)= \eta_{e}^{R}(M_{1},N_1) 
\text{ and } \\
\eta_{e-1}^{S}(E,F) & = - \eta_{e-1}^{S}(M_{1},F) = \eta_{e-1}^{S}(M_{1},N_{1})\,.
\end{split}
\end{equation*}
Our assumption that $e\geq \max\{2,\codim S+1\}$, together with
\cite[Theorem 4.1(3)]{Da2}, allow us to invoke \cite[Theorem
4.3(3)]{Da2}, which says that
\[
2e \cdot \eta_e^R(M_{1},N_{1}) = \eta^S_{e-1}(M_{1},N_{1})\,.
\]
This gives (ii), completing the proof.
\end{proof}

The next theorem is our main result.  As its hypotheses 
are technical, several of its consequences are discussed in section
\ref{sec:van-eta}; see section \ref{sec:prelims} for background.   

\begin{theorem} 
\label{thm:eta2} 
Let $R$ be a local ring whose completion is a complete intersection in an \emph{unramified} regular local ring, of relative codimension $c\geq 1$. Let $M, N$ be finitely generated $R$-modules.  Assume the following hold:
\begin{enumerate}[\rm(i)]
\item $\dim(R)\ge c$.
\item The pair $(M,N)$ satisfies $\SP c$. 
\item $\Supp_R (\tp RN) \subseteq \Supp_R(M)$.
\item $\eta_{c}^{R}(M,N)=0$
\end{enumerate}
Then $\Tor^{R}_{i}(M,N)=0$ for all $i\ge 1$.
\end{theorem}

\begin{proof} 
The case $c=1$ is Lemma \ref{thm:eta}.  For $c\ge2$, proceed by induction on $c$. We can assume $R$ is complete, so that $R=Q/(\ul{f})$, where $Q$ is an unramified regular local ring and $\ul f=f_{1},\dots,f_{c}$ is a $Q$-regular sequence; see (\ref{subsec:ram})  and (\ref{lem:ascent}). Let $R=S/(f)$, where $S=Q/(f_1,\dots,f_{c-1})$ and $f=f_c$.

Hypothesis (ii) implies $\Tor^{R}_{i}(M,N)$ has finite length for all $i\gg 0$; see (\ref{not:SP}). Hence Corollary~\ref{cor:ci-van} implies that, for all primes $ \fp $ with $ \height(\fp) \leq c-1$, 
\begin{equation}
\label{eq:tor0}
\Tor^{R}_{i}(M,N)_{\fp} = 0 \text{ for all $i\ge 1$}.  
\end{equation}

Condition (ii) also implies $M$ and $N$ are torsion-free since $ c \geq 2 $; see (\ref{not:SP}). Hence quasi-liftings $E$ and $F$ of $M$
and $N$ to $S$ exist; see (\ref{pushforward}). Using the vanishing of Tors in \eqref{eq:tor0} and \cite[Theorem 4.8]{HJW}  (cf. \cite[Proposition 3.1(7)]{Ce}), one gets that
\begin{equation}
\label{eq:codim-red}
E\otimes_SF \text{ satisfies } (S_{c-1}) \text{ as an $S$-module.}
\end{equation}
It follows from \cite[Propositions 1.6 and 1.7]{HJW} (see also
\cite[Proposition 3.1(2) and 3.1(6)]{Ce}) that the assumptions in (i)
of $\SP c$ pass to $E$ and $F$; see (\ref{not:SP}).
\begin{equation}
\label{eq:Serre-free}
E \text{ and } F \text{ satisfy } (S_{c-1}) \text{ as $S$-modules}\,.
\end{equation}
Lemma \ref{prop:eta} guarantees that $\Tor^S_i(E,F)$ has finite length
for all $i\gg0$ and that $\eta_{c-1}(E,F) = 0$. In particular the pair $E,
F$ satisfies $\SP {c-1}$ over the ring $S$.  Moreover, $E$ and $F$,
being syzygies, are torsion-free, so we indeed have that $\Supp_S(\tp
SF)\subseteq \Supp_S(E)$. Now the inductive hypothesis implies that
\begin{equation}\label{eq:Tor-S-van}
\Tor^S_i(E,F) =0 \text{ for all } i\geq 1. 
\end{equation}
Condition (ii) also implies that $M\otimes_RN$ is reflexive since $ c \geq 2 $; see (\ref{Serre}). Further $\Tor^{R}_{i}(M,N)_{\fp} = 0$ for all $i\ge 1$ and for all $\fp\in \Spec(R)$ with $\height(\fp)\leq 1$; see \eqref{eq:tor0}. Thus Proposition~\ref{prop:push-lift2} and (\ref{eq:Tor-S-van}) yield
$\Tor^R_i(M,N)=0$ for all $i \ge1$.
\end{proof}

\begin{remark} In Theorem \ref{thm:eta2}, if $c\geq 2$, hypothesis (ii) implies that $N$ is torsion-free, i.e., $\tp RN=0$; see (\ref{Serre}) and (\ref{not:SP}). Thus, when $c\geq 2$, hypothesis (iii) of Theorem \ref{thm:eta2} is redundant.
\end{remark}

When $\dim(R) > c$, the equivalence of
(i) and (ii) in the following corollary seems interesting; see also (\ref{dfor}).  
Actually, in that case the equivalence of (ii) and (iii) holds without
the assumption that $\eta_c^R(M,N) = 0$.  See \cite[Corollary 2.4]{Ce}.

\begin{corollary} \label{cor:MCM} %
  Let $R$ be an isolated singularity whose completion is a complete
  intersection in an unramified regular local ring, of relative
  codimension $c$. Let $M$ and $N$ be maximal Cohen-Macaulay
  $R$-modules. Assume $\dim(R) \ge c$. Assume further that
  $\eta_c^R(M,N) = 0$.  The following conditions are equivalent:
\begin{enumerate}[\rm(i)]
\item $M\otimes_RN$ satisfies $(S_c)$.
\item $M\otimes_RN$ is maximal Cohen-Macaulay.
\item $\Tor_i^R(M,N)=0$ for all $i\geq 1$, and hence
the depth formula holds.  
\end{enumerate}
\end{corollary}

Over a complete intersection, vanishing of Ext is closely related to vanishing of Tor:
$\Ext^{i}_{R}(M,N)=0$ for all $i\gg 0$ if and only if $\Tor^R_i(M,N)=0$ for all $i\gg 0$; see \cite[Remark
6.3]{AvBu}. Our next example shows the hypotheses of Theorem
\ref{thm:eta2} do {\em not} force the vanishing of $\Ext^{i}_{R}(M,N)$
for all $i\geq 1$.

\begin{example} \label{exExt} %
  Let $(R, \fm, k)$ be a complete intersection with $\codim(R)=2$ and
  $\dim(R) \geq 3$. Let $N$ be the $d$th syzygy of $k$, where
  $d=\dim(R)$, and let $M$ be the second syzygy of
  $R/(\underline{x})$, where $\underline{x}$ is a maximal $R$-regular
  sequence.

  Note that $N$ is maximal Cohen-Macaulay, $\depth(M)=2$ and $N_{\fp}$
  is free over $R_{\fp}$ for all primes $\fp\neq \fm$. It follows,
  since $\pd(M)<\infty$, that $\eta^{R}_{2}(M,N)=0$ and
  $\Tor^{R}_{i}(M,N)=0$ for all $i\geq 1$; see (\ref{chunk:eta}) and
  Theorem \ref{prop:jorgensen}. Therefore the depth formula
  (\ref{dfor}) shows that $\depth(M\otimes_{R}N)=2$. Since $M$ is a
  second syzygy, it satisfies $(S_2)$ and hence $M\otimes_{R}N$
  satisfies $(S_2)$; see (\ref{Serre}). In particular, the pair $M,N$
  satisfies $\SP 2$; see (\ref{not:SP}). However
  $\Ext_{R}^{d-2}(M,N) = \Ext^d(R/(\underline{x}),N) \neq 0$; 
  see, for example, \cite[Chapter 19,
  Lemma 1(iii)]{Mat}.
\end{example}

Here is the extension of Dao's theorem \cite[Theorem 7.7]{Da2} promised
in the introduction (cf. Theorem~\ref{Long's result}):

\begin{corollary}
\label{cor:dao} 
Let $R$ be a local ring that is a complete intersection, and let $M$ and $N$ be finitely generated $R$-modules. Assume that the following conditions hold for some integer $e\ge \codim(R)$:
\begin{enumerate}[\rm(i)]
\item $M$ and $N$ satisfy $(S_{e})$.
\item $M\otimes_{R}N$ satisfies $(S_{e+1})$.
\item $M_{\fp}$ is a free for all prime ideals $\fp$ of $R$ of height
  at most $e$. 
\end{enumerate}
Then $\Tor^{R}_{i}(M,N)=0$ for all $i\ge 1$ and hence the depth formula holds. 
\end{corollary}

\begin{proof}
  If $e = 0$ this is the theorem of Auslander \cite{Au} and 
  Lichtenbaum \cite[Corollary 2]{Li}.  
  Assume now that $e \ge 1$.  We use induction on $\dim R$.  If $\dim R \le e$, condition (iii) implies that
  $M$ is free, and there is nothing to prove.  Assuming $\dim R \ge e+1$, we note that the  hypotheses localize, so
  $\Tor^{R}_{i}(M,N)_{\fp}=0$ for each $i\ge 1$ and each prime ideal $\fp$ in the
  punctured spectrum of $R$; that is to say,
  $\Tor^{R}_{i}(M,N)$ has finite length for all $i\ge 1$. Thus the pair $M,N$
  satisfies $\SP{e+1}$. Moreover, since $\codim R < e+1$, we have $\eta^R_{e+1} = 0$ by
 item (iv) of  \eqref{chunk:eta}.  The
  completion of $R$ can be realized as a complete intersection, of
  relative codimension $e+1$, in an unramified regular local ring (see \ref{subsec:ram}). Hence
  the desired result follows from Theorem~\ref{thm:eta2}.
\end{proof}

\section{Vanishing of $\eta$} \label{sec:van-eta}

In this section we apply our results to situations where the $\eta$-pairing is known to vanish.  We know, from Theorem \ref{thm:eta2}, that, as long as the critical
hypothesis $\eta^{R}_{c}(M,N)=0$ holds, we can replace $c$ with $c-1$
in the hypotheses of Theorem~\ref{Long's result} and still conclude
the vanishing of Tor. Although it is not easy to verify vanishing
of $\eta$ (see Conjectures \ref{conjD}), there are several classes of
rings $R$ for which it is known that $\eta^{R}(M,N)=0$ for all finitely generated
$R$-modules $M$ and $N$. For example, if $R$ is an even-dimensional
simple (``ADE'') singularity in characteristic zero, then Dao
\cite[Corollary 3.16]{Da1} observed that $\theta^R(M,N) = 0$; see
\cite[Corollary 3.6]{Da1} and also \cite[\S3]{Da1} for more
examples.

Now we give a localized version of a vanishing theorem for graded rings, due to Moore, Piepmeyer, Spiroff, and Walker \cite{MPSW2}.

\begin{proposition}
\label{prop:MPSW-local}
Let $k$ be a perfect field and $Q=k[x_{1}, \dots, x_{n}]$ the polynomial ring with the standard grading. Let
$\ul{f}=f_{1},\dots, f_{c}$ be a $Q$-regular sequence of homogeneous polynomials, with $c\geq 2$.  Put $A=Q/(\ul{f})$ and   $R= A_\fm$, where $\fm = (x_1,\dots,x_n)$.  Assume that  $A_\fp$ is a regular local ring for each $\fp$ in $\Spec(A)\!\setminus\!\{\fm\}$.  Then $\eta_c^R(M,N) = 0$ for all finitely generated $R$-modules $M$ and $N$. In particular, if $n\ge 2c$ and the pair $M,N$ satisfies $\SP c$, then $M$ and $N$ are Tor-independent.
\end{proposition}

\begin{proof}
Choose finitely generated $A$-modules $U$ and $V$ such that $U_\fm \cong M$ and $V_\fm \cong N$. For any maximal ideal $\fn\ne \fm$, the local ring $A_\fn$ is regular and  hence $\Tor^A_i(U,V)_\fn = 0$ for $i\gg 0$.  It follows that the map $\Tor^A_i(U,V) \to \Tor^R_i(M,N)$ induced by the
localization maps $U \to M$ and $V\to N$ is an isomorphism for $i\gg0$.  Also, for any $A$-module supported at $\fm$, its length as an $A$-module is equal to its length as an $R$-module. In conclusion, 
$\eta_c^R(M,N) = \eta_c^A(U,V)$.

As $k$ is perfect, the hypothesis on $A$ implies that the $k$-algebra $A_{\fp}$ is smooth for each non-maximal prime $\fp$ in $A$;  see \cite[Corollary 16.20]{EiBook}. Thus, the  morphism of schemes $\Spec(R)\!\setminus\!\{\fm\} \to \Spec(k)$ is smooth.   Now \cite[Corollary 4.7]{MPSW2} yields
$\eta_c^A(U,V) = 0$, and hence
$\eta_{c}^{R}(M,N)=0$. It remains to note that if $n\ge 2c$, then $\dim R\ge c$, so Theorem~\ref{thm:eta2} applies.
\end{proof}

Next, we quote a recent theorem due to Walker; it provides strong support for Conjectures~\ref{conjD}, at least in equicharacteristic zero.

\begin{theorem}\emph{(Walker~\cite[Theorem 1.2]{Walker})}
\label{thm:Walker} %
 Let $k$ be a field of characteristic zero, and let $Q$ a smooth
  $k$-algebra.  Let  $\ul{f}=f_1, \dots, f_c$ be a $Q$-regular
  sequence, with $c\geq 2$, and put $A = Q/(f_1, \dots, f_c)$.  Assume the singular locus
$\{\fp \in \Spec(A): A_{\fp} \text{ is not regular} \}$ is a finite set of
maximal ideals of $A$.
Then $\eta_c^A(U,V) = 0$ for all
finitely generated $A$-modules $U$, $V$.  
\end{theorem}

\begin{corollary} \label{cor:Walker-local} %
  With $A$ as in \ref{thm:Walker},  put $R= A_\fm$ where
  $\fm$ is any maximal ideal of $A$.  Then 
$\eta_c^R(M,N) = 0$ for all finitely generated $R$-modules $M$ and $N$. 
In particular, if $\dim R \ge c$ and the pair $M,N$ satisfies $\SP c$, then $M$ and $N$ are Tor-independent.
 
\end{corollary}

\begin{proof} By inverting a suitable element of $Q$, we may assume that $A_\fp$ is a regular local ring
for every prime ideal $\fp\ne \fm$.  Now proceed as in the first paragraph of the proof of
Proposition~\ref{prop:MPSW-local}.
\end{proof}

\begin{theorem} \label{coreta0} %
  Let $(R,\fm,k)$ be a two-dimensional, equicharacteristic, normal,
  excellent complete intersection of codimension $c$, with
  $c\in\{1,2\}$, and let $M$ and $N$ be finitely generated
  $R$-modules. Assume $k$ is contained in the algebraic closure of a
  finite field. Assume further that $M,N$ satisfy the
  conditions (i) and (ii) of $\SP c$. Then $\Tor^{R}_{i}(M,N)=0$ for
  all $i\ge 1$.
\end{theorem}

\begin{proof} The completion $\widehat R$ is an isolated singularity
  because $R$ is excellent; see \cite[Proposition 10.9]{LW}, and
  so $\widehat R$ is a normal domain.  Replacing $R$ by
  $\widehat R$, we may assume that $R=S/(\ul f)$, where $(S,\fn,k)$ is
  a regular local ring and $\ul f$ is a regular sequence in $\fn^2$ of
  length $c$.  Let $\overline k$ be an algebraic closure of $k$, and
  choose a \emph{gonflement} $S\hra (\overline S,\overline \fn,
  \overline k)$ lifting the field extension $k\hra \overline k$; see
  \cite[Chapter 10, \S3]{LW}.  This is a flat local homomorphism and
  is an inductive limit of \'etale extensions.  Moreover, $\fn
  \overline S = \overline \fn$, so $\overline S$ is a regular local
  ring.  By \cite[Proposition 10.15]{LW}, both $\overline S$ and
  $\overline R := \overline S/(\ul f)$ are excellent, and $\overline
  R$ is an isolated singularity. Therefore $(\overline R, \overline
  \fm, \overline k)$ is a normal domain.  Finally, we pass to the
  completion $\widehat{S}$ of $\overline S$ and put $\Lambda = \widehat{S}/(\ul
  f)$.  This is still an isolated singularity, a normal domain, and a
  complete intersection of codimension $c$. Moreover, our hypotheses
  on $M$ and $N$ ascend along the flat local homomorphism $R\to
  \Lambda$; see (\ref{lem:ascent}).  Since $\Lambda$ is an isolated singularity,
  $\Tor_i^\Lambda(\Lambda\otimes_RM,\Lambda\otimes_RN)$ has finite length for
  $i\gg0$; thus the pair $\Lambda\otimes_RM$, $\Lambda\otimes_RN$ satisfies $\SP c$.

  It follows from \cite[Proposition 2.5 and Remark 2.6]{CeD} that
  $G(\Lambda)/L$ is torsion, where $G(\Lambda)$ is the Grothendieck
  group of $\Lambda$ and $L$ is the subgroup generated by classes of
  modules of finite projective dimension. This implies that
  $\eta_c^{\Lambda}(\Lambda\otimes_RM,\Lambda\otimes_RN) = 0$; see
 \cite[Corollary 3.1]{Da1} and the
  paragraph preceding it. Now Theorem \ref{thm:eta2} implies
  that $\Tor_i^{\Lambda}(\Lambda\otimes_RM,\Lambda\otimes_RN) = 0$ for
  all $i\ge 1$: the requirement on supports is automatically
  satisfied, since $\Lambda$ is a domain; see Remark
  \ref{rmks:Supp}(i). Faithfully flat descent completes the
  proof.
\end{proof}

\appendix
\section{An application of pushforwards} 

In Theorem \ref{thm:tor1} we use pushforwards to generalize a theorem due to Celikbas \cite[Theorem 3.16]{Ce}. We have two preparatory
results. The first one is a special case of a theorem of Jorgensen:

\begin{theorem} {\em (\cite[Theorem 2.1]{Jo1})}\label{prop:jorgensen} %
Let $R$ be a complete intersection and let $M$ and $N$ be finitely generated $R$-modules. Assume $M$ is maximal Cohen-Macaulay. If
$\Tor_i^R(M,N) = 0$ for all $i\gg0$, then $\Tor^{R}_{i}(M,N)=0$ for all $i\ge 1$.
\end{theorem}

\begin{corollary}
\label{cor:tortor} 
  Let $R$ be a  complete intersection and let $M, N$ be finitely
  generated $R$-modules. If $\Tor_i^R(M,N)$ is torsion for all $i\gg 0$, then
  $\Tor_i^R(M,N)$ is torsion for all $i\geq 1$.
\end{corollary}

\begin{proof} Let $\fp$ be a minimal prime ideal of $R$. By  (\ref{ts}), it suffices to prove
  that $\Tor_i^{R_\fp}(M_\fp,N_\fp) = 0$ for all $i\ge 1$.
For that we may assume $M_{\fp}\neq 0$. Then, since $R_{\fp}$ is artinian, it
follows that $M_\fp$ is a maximal Cohen-Macaulay
  $R_{\fp}$-module. Therefore Theorem \ref{prop:jorgensen} gives the
  desired vanishing.
\end{proof}

\begin{corollary}
\label{cor:ci-van}  
Let $R$ be a complete intersection, and let $M, N$ be finitely  generated $R$-modules.  Assume $M$ satisfies
  $(S_w)$, where $w$ is a positive integer, and that $\Tor^R_i(M,N)$ has finite length for
  all $i\gg0$. Let $\fp$ be a non-maximal prime
  ideal of $R$ such that $\height(\fp) \leq w$. Then $\Tor^{R}_{i}(M,N)_{\fp} = 0$ for all $i\ge 1$.
\end{corollary}

\begin{proof} 
Serre's condition $(S_{w})$ localizes, so  $M_{\fp}$ is either zero or a
 maximal Cohen-Macaulay $R_\fp$-module; see (\ref{Serre}). As
$\Tor_i^{R_\fp}(M_\fp, N_\fp) = 0$ for $i\gg 0$, Theorem
\ref{prop:jorgensen} implies that $\Tor_i^{R_\fp}(M_\fp, N_\fp) = 0$
for all $i\geq 1$.
\end{proof}

The next theorem generalizes a result due to Celikbas \cite[3.16]{Ce}; we emphasize that the ambient regular local ring in Theorem \ref{thm:tor1} is allowed to be ramified.

\begin{theorem}
\label{thm:tor1}
Let $R$ be a complete intersection  with $\dim R\ge \codim R$, and let
$M$ and $N$ be finitely generated $R$-modules. Assume the pair $M, N$ satisfies $\SP {c}$ for some $c\geq \codim R$. If $c=1$, assume further that $M$ or $N$ is torsion-free. If $\Tor^{R}_{1}(M,N)=0$, then $\Tor^{R}_{i}(M,N)=0$ for all $i\ge 1$.
\end{theorem}

\begin{proof} %
Without loss of generality, one may assume that $c=\codim R$. When $c=0$, the desired result is the rigidity theorem of Auslander \cite{Au} and Lichtenbaum \cite{Li}, so in the remainder of the proof we assume that $c\ge 1$. 

Assume first that $c=1$. By hypotheses $\Tor^{R}_{i}(M,N)$ has finite length for $i\gg 0$ and $M\otimes_{R}N$ is torsion-free; see (\ref{not:SP}).
Moreover, we may assume $N$ (say) is torsion-free. Tensoring $M$ with the pushforward   (\ref{pushforward}) for $N$ gives the following:
 \begin{gather} 
  \label{eq:tor11inj} %
   \Tor^{R}_{1}(M_{},N_{1}) \hookrightarrow M_{} \otimes_{R} N \\
\label{eq:tor11iso} %
   \Tor^{R}_{i}(M_{},N_{1})\cong \Tor^{R}_{i-1}(M_{},N) \text{
     for all $i\ge 2$.}
 \end{gather}
 Equation (\ref{eq:tor11iso}) implies that $\Tor_i^R(M_{},N_1)$ has finite
 length for all $i\gg 0$. Therefore, since $\dim(R) \geq 1$,
 $\Tor_i^R(M_{},N_1)$ is torsion for all $i\gg 0$; see (\ref{ts}). Now
Corollary~\ref{cor:tortor} implies that $\Tor^{R}_{i}(M_{},N_{1})$ is
 torsion for all $i\geq 1$. As $M\otimes_{R}N$ is torsion-free, we
 deduce from (\ref{eq:tor11inj}) that $\Tor^{R}_{1}(M,N_{1})=0$. By
 (\ref{eq:tor11iso}) we have $\Tor^{R}_{2}(M_{},N_{1})\cong
 \Tor^{R}_{1}(M_{},N)=0$. Therefore $\Tor^{R}_{2}(M,N_{1})=0=
 \Tor^{R}_{1}(M,N_{1})$, and hence Murthy's rigidity theorem
 \cite[Theorem 1.6]{Mu} implies that $\Tor^{R}_{i}(M_{},N_{1}) = 0$ for
 all $i\ge 1$.  Now (\ref{eq:tor11iso})
 completes the proof for the case  $c = 1$.

Assume now that $c\geq 2$. We define a sequence $M_{0},
M_{1},\dots,M_{c-1}$ of finitely generated modules  by setting $M_{0}=M$, and $M_{n}$ to be the
pushforward of $M_{n-1}$, for all $n=1,\dots,c-1$. These pushforwards
exist: $M_{0}$ satisfies $(S_{c-1})$ by hypothesis (\ref{not:SP})(i), and so, by Proposition~\ref{prop:push-lift}(i), 
\begin{enumerate}
\item 
each $M_{n}$ satisfies $(S_{c-n-1})$. 
\end{enumerate} 
For the desired result, it suffices to prove that
$\Tor^{R}_{i}(M_{c-1},N)=0$ for all $i\ge c$.  We will, in fact, prove
this for all $i\ge 1$.  To this end, we establish by induction  
that the following hold for $n = 0, \ldots, c-1$:
\begin{enumerate}\setcounter{enumi}{1}
\item
 $M_{n}\otimes_{R}N$ satisfies ($S_{c-n}$);
\item
$\Tor^{R}_{i}(M_{n},N)$ has finite length for all $i\gg 0$;
\item
$\Tor^{R}_{i}(M_{n},N)=0$ for $i=1,\dots,n+1$\,.
\end{enumerate}

For $n=0$, conditions (2) and (3) are part of (\ref{not:SP}), while (4) is from our hypothesis that $\Tor^{R}_{1}(M,N)=0$; recall that $M_0=M$. Assume that (2), (3) and (4) hold for some integer $n$ with $0\leq n \leq c-2$.

Tensor the pushforward of $M_n$ with $N$, see (\ref{pushforward}), to 
obtain 
\begin{equation} 
\label{eq:tor12}
\Tor^{R}_{i}(M_{n+1},N)\cong \Tor^{R}_{i-1}(M_{n},N) \text{ for all
  $i\ge 2$},
\end{equation}
and the following exact sequence in which $ F $ is finitely
generated and free: 
\begin{equation} 
 \label{eq:tor13}
0\to \Tor^{R}_{1}(M_{n+1},N) \to M_{n}\otimes_{R} N \to
F\otimes_{R} N \to M_{n+1}\otimes_{R}N \to  0 \, . 
\end{equation} 

Induction and \eqref{eq:tor12} imply that $\Tor^{R}_{i}(M_{n+1},N)$
has finite length for all $i\gg0$, so (3) holds; furthermore, by
Corollary~\ref{cor:tortor}, $\Tor^{R}_{i}(M_{n+1},N)$ is torsion for
all $i\ge 1$. (Recall that $\dim(R) \geq \codim(R)=c \geq 1$ so that finite
length modules are torsion.)  Since $n\le c-1$, condition (2) implies
that $M_{n}\otimes_{R}N$ satisfies ($S_{1}$) and hence
$M_{n}\otimes_{R}N$ is torsion-free; therefore the exact sequence
\eqref{eq:tor13} forces $\Tor^{R}_{1}(M_{n+1},N)$ to vanish. Now
\eqref{eq:tor12} gives (4). It remains to verify (2), namely, that 
$M_{n+1}\otimes_RN$ satisfies $(S_{c-n-1})$.  To that end, let
$\fp\in \Supp(M_{n+1}\otimes_{R}N)$.  We will verify that 
$\depth_{R_{\fp}}(M_{n+1}\otimes_{R}N)_{\fp}\geq \min\{c-n-1,
\height(\fp)\}$; see (\ref{Serre}).

Suppose $\height(\fp) \ge c-n$. Recall, by hypothesis (\ref{not:SP})(i), $N$
satisfies ($S_{c-1}$). Hence $F\otimes_{R}N$, a direct sum of copies
of $N$, satisfies (S$_{c-n-1})$. In particular it follows that
$\depth_{R_\fp} (F\otimes_{R}N)_\fp \ge c-n-1$. Furthermore, by (2) of
the induction hypothesis, we have that $\depth_{R_\fp}
(M_n\otimes_RN)_\fp \ge c-n$. Recall that
$\smash{\Tor^{R}_{1}(M_{n+1},N)=0}$. Therefore, localizing the short exact
sequence in \eqref{eq:tor13} at $\fp$, we conclude by the depth lemma
that $\depth_{R_{\fp}}(M_{n+1}\otimes_{R}N)_{\fp}\geq c-n-1$.

Next assume $\height(\fp) \leq c-n-1$. We want to show that $(M_{n+1}\otimes_RN)_\fp$ is maximal Cohen-Macaulay. By the induction hypotheses, $\Tor^{R}_{i}(M_{n},N)$ has finite length for all $i\gg 0$. As $n\ge 0$, we see that $\dim(R) \ge \codim(R)=c \geq c-n$, whence $\fp$ is not the maximal ideal. Thus $\Tor^{R}_{i}(M_{n},N)_{\fp}=0$ for all $i\gg 0$. Now, setting $w=c-n-1$ and using Corollary~\ref{cor:ci-van} for the pair $M_{n},N$, we conclude that $\Tor^{R}_{i}(M_{n},N)_{\fp}=0$ for all $i\geq 1$.  Then \eqref{eq:tor12} and the already established fact that $\Tor^{R}_{1}(M_{n+1},N)=0$ give $\Tor^{R}_{i}(M_{n+1},N)_{\fp}=0$ for all $i\ge 1$.  Thus the depth formula holds; see \eqref{dfor}:
\begin{equation*}
  \depth_{R_{\fp}}(M_{n+1})_{\fp}+\depth_{R_{\fp}}(N_{\fp})
  =\depth(R_{\fp})+\depth_{R_{\fp}}(M_{n+1}\otimes_{R}N)_{\fp}\,.  
\end{equation*}
Since Serre's conditions localize, $N_{\fp}$ is maximal Cohen-Macaulay
over $R_{\fp}$; see hypothesis (\ref{not:SP})(i). Also, $(M_{n+1})_{\fp}$ is maximal
Cohen-Macaulay whether or not $(M_{n})_{\fp}$ is zero; see the
pushforward sequence or Proposition~\ref{prop:push-lift}(ii).  By the
depth formula, $(M_{n+1}\otimes_RN)_\fp$ is maximal Cohen-Macaulay.  
Thus $M_{n+1}\otimes_{R}N$ satisfies (2), and the induction is complete.

Now we parallel the argument for the case $c=1$. At the end,
$\Tor^{R}_{i}(M_{c-1},N)$ has finite length for all $i\gg0$, and is
equal to $0$ for $i=1,\dots,c$. Tensoring $M_{c-1}$ with the pushforward of $N$, we get
\begin{gather} 
 \label{eq:tor14}
 \Tor^{R}_{i}(M_{c-1},N_{1})\cong 
 \Tor^{R}_{i-1}(M_{c-1},N) \text{ for all } i\ge 2,
\\
 \label{eq:tor15} {\text{and \quad }} 
\Tor^{R}_{1}(M_{c-1},N_{1}) \hookrightarrow M_{c-1} \otimes_{R} N. 
\end{gather}
In view of (\ref{eq:tor14}), it suffices to show that
$\Tor^{R}_{1}(M_{c-1},N_{1})=0$: this will imply
$\Tor^R_i(M_{c-1},N_1) = 0$ for all $i=1, \dots, c+1$, and hence
Murthy's rigidity theorem \cite[Theorem 1.6]{Mu} will yield that
$\Tor^{R}_{i}(M_{c-1},N_{1}) = 0$ for all $i\ge 1$, and consequently
$\Tor^{R}_{i}(M_{c-1},N) = 0$ for all $i\ge 1$ by \eqref{eq:tor14}. We
know that $M_{c-1}\otimes_{R}N$ is torsion-free. Therefore we use
(\ref{eq:tor15}) and Corollary~\ref{cor:tortor}, and obtain
$\Tor^{R}_{1}(M_{c-1},N_{1})=0$, as we did in the case $c=1$.
\end{proof}

\section{Amending the literature}
\label{sec:appendix}

We use Theorem~\ref{thm:tor1} to give a different proof of an
important result of Huneke and Wiegand; see Theorem
\ref{prop:second-rigidity} and the ensuing paragraph. We also point
out a missing hypothesis in a result of C.\ Miller \cite[Theorem
3.1]{Mi}, and state the  corrected form of her theorem in Corollary
\ref{cor:hyp-power}.  At the end of the paper we indicate an alterate route to the proof of 
the following result
\cite[Theorem 3.1]{HW1}, the main theorem of the 1994 paper of Huneke and Wiegand:

\begin{theorem}[Huneke and Wiegand \cite{HW1}]\label{prop:HW1-main} %
  Let $R$ be a hypersurface, and let $M$ and $N$ be finitely generated
  $R$-modules. If $M$ or $N$ has rank, and $M\otimes_RN$ is maximal
  Cohen-Macaulay, then both $M$ and $N$ are maximal Cohen-Macaulay,
  and either $M$ or $N$ is free.
\end{theorem}

Theorem \ref{prop:HW1-main} and its variations  have been analyzed, used, and studied in the literature; see \cite{CeW} and \cite{surveydao} for
some history and many consequences of the theorem. The following result \cite[Theorem 2.7]{HW1} played an important role in its proof.

\begin{theorem}[Huneke and Wiegand \cite{HW1}] 
\label{prop:second-rigidity}
Let $R$ be a hypersurface and let $M,N$ be nonzero finitely  generated $R$-modules. Assume $M\otimes_{R}N$ is reflexive and that $N$ has  rank.  Then the following conditions hold:
\begin{enumerate}[\rm(i)]
\item\label{eq:deadTor}
$\Tor^{R}_{i}(M,N)=0$ for all $i\ge 1$.
\item\label{eq:refl-tf}$M$ is reflexive, and $N$ is torsion-free.
\end{enumerate}
\end{theorem}

Theorem \ref{prop:second-rigidity} was established by Huneke and Wiegand in \cite[Theorem 2.7]{HW1}: however their conclusion was that
{\em both} $M$ and $N$ are reflexive, and  the proof of this stronger claim is flawed. Dao realized the oversight of \cite[Theorem 2.7]{HW1}, and Huneke and Wiegand addressed it in the erratum \cite{HW1e}. A similar flaw can be found in Miller's paper; see \cite[Theorems 1.3 and 1.4]{Mi} and compare it with our correction in Corollary \ref{cor:hyp-power}. The version stated above reflects our current understanding and is from the paper \cite{CeP}. We do not yet know whether $N$ is forced to be reflexive, that is, the question below remains open; cf. \cite[Theorem 2.7]{HW1} and \cite[Theorem 1.3]{Mi}.

\begin{question} 
\label{Gregproveit} 
Let $R$ be a hypersurface and $M,N$ nonzero finitely generated $R$-modules. If  $N$ has rank and $M\otimes_{R}N$ is reflexive, must  \emph{both} $M$ and $N$ be reflexive?
\end{question}
This question has been recently studied in \cite{CeP}, which gives partial answers using the New Intersection Theorem.

We now show how Theorem~\ref{prop:second-rigidity} follows from Theorem~\ref{thm:tor1}. In fact, one needs only the case 
$c=1$ of Theorem~\ref{thm:tor1}.

\begin{proof}[Proof of Theorem~\ref{prop:second-rigidity} using
  Theorem \ref{thm:tor1}] Set $d=\dim R$. If $d = 0$, then $N$
  is free (since it has rank), so all is well.  From now on assume 
  $d\ge1$.
 We remark at the outset that neither $M$ nor
  $N$ can be torsion, i.e., $\tf R M\ne 0$ and $\tf R N\ne 0$. Also,
  by the assumption of rank, $\Supp(N) = \Spec(R)$.  Suppose first
  that both $M$ and $N$ are torsion-free; we will prove
  \eqref{eq:deadTor} by induction on $d=\dim R$.  
  Let $M_{1}$ denote the pushforward of $M$; see
  (\ref{pushforward}).  Then $\Tor^R_1(M_1,N)$ is torsion as $N$
  has rank.  Since $M\otimes_RN$ is torsion-free, applying $ {-}
  \tensor_R N $ to 
  \eqref{eq:pushforward} shows that   
\begin{equation}\label{eq:snort}
  \Tor^R_1(M_1,N)=0\,. 
\end{equation}

Suppose for the moment that  $d=1$.  Since $N$ has rank, there is an exact sequence
\[
0\to N \to F \to C \to 0\,,
\]
in which $F$ is free and $C$ is torsion. (See \cite[Lemma 1.3]{HW1}.)  Note that $C$ is of finite length since $d=1$. Note also that $\Tor_2^R(M_1,C) \cong \Tor_1^R(M_1,N) = 0$; see \eqref{eq:snort}. Therefore \cite[Corollary 2.3]{HW1} implies that $\Tor^R_i(M_1,C)=0$ for all $i\ge 2$,  and hence $\Tor_i^R(M_1,N)=0$ for all $i\ge1$. Now (\ref{eq:pushforward}) establishes (i).

Still assuming that both $M$ and $N$ are torsion-free, let $d\ge2$. The inductive hypothesis implies that $\Tor_i^R(M,N)$ has finite length for all $i\ge 1$. In particular $\Tor_{i}^{R}(M,N)_{\fq}=0$ for all prime ideals $\fq$ of $R$ of height at most one. Therefore Proposition~\ref{prop:push-lift2}
shows that $M_1\otimes_RN$ is torsion-free, that is, $M_1\otimes_RN$ satisfies $(S_1)$; see (\ref{TRM}) and (\ref{Serre}). Furthermore, from the pushforward exact sequence \eqref{eq:pushforward}, we see that $\Tor_i^R(M_1,N)$ has finite length for all $i\ge 2$. Consequently the pair $M_1$, $N$ satisfies $\SP{1}$. Now Theorem~\ref{thm:tor1}, applied to $M_1$, $N$, shows that $\Tor_i^R(M_1,N)=0$ for all $i\ge1$.  By
\eqref{eq:pushforward}, we see that $\Tor_i^R(M,N)=0$ for all $i\ge 1$. This proves \eqref{eq:deadTor} under the additional assumption that $M$ and $N$ are torsion-free.

Since $M\otimes_{R}N$ is torsion-free, it follows from (\ref{Noproof}) that there are isomorphisms
\[
M\otimes_{R}N \cong M\otimes_{R}\tf RN \cong \tf RM\otimes_{R} N \cong \tf RM\otimes_{R}\tf RN\,.
\]
In particular,  $\tf RM\otimes_{R}\tf RN$  is also reflexive.  As noted before, neither $M$ nor $N$ is torsion so $\tf RM$ and $\tf RN$ are nonzero. As $N$ has rank so does $\tf RN$,  so the already established part of the result (applied to $\tf RM$ and $\tf RN$) yields 
\[
\Tor^{R}_{i}(\tf RM,\tf RN)=0\quad\text{for $i\ge 1$.}
\]
Given this, since $\tf RM\otimes_{R}N$ is torsion-free by the isomorphisms above, applying (\ref{Noproof}) to the $R$-modules $\tf RM$ and $N$ gives  $N=\tf RN$; then applying (\ref{Noproof}) to $M$ and $N$ yields $M=\tf RM$. In conclusion, $M$ and $N$ are torsion-free, and hence $\Tor_i^R(M,N)=0$ for all $i\ge1$.  From the last, the depth formula holds.  

The remaining step is to prove that $M$ is reflexive.  Since $\Supp(N)
= \Spec(R)$, we have $\depth(N_\fp) \le \height(\fp)$ for all primes
$\fp$ of $R$.  Localizing the depth formula \eqref{dfor} shows Serre's
condition $(S_2)$ on $M$;  see (\ref{Serre}).  
\end{proof}

The next result is due to C.~Miller. In her paper \cite{Mi}, the
essential requirement --- that $M$ have rank --- is missing: for
example, the module $M=R/(x)$ over the node $k[\![x,y]\!]/(xy)$ is not
free, yet $M\otimes_RM$, which is just $M$, is maximal Cohen-Macaulay
and hence reflexive.  We state her result here in its corrected
form and include a proof for completeness.

\begin{corollary}{\em (C.~Miller \cite[Theorem
    3.1]{Mi})} \label{cor:hyp-power} %
  Let $R$ be a $d$-dimensional hypersurface and let $M$ a finitely
  generated $R$-module with rank.  If $\otimes_R^nM$ is reflexive for
  some $n\ge \max\{2, d-1\}$, then $M$ is free.
\end{corollary}

\begin{proof} If $d\le 2$, then $\otimes_R^nM$ is maximal
  Cohen-Macaulay, and Theorem~\ref{prop:HW1-main} gives the result.
  Assume now that $d\ge 3$.  Applying Theorem
  ~\ref{prop:second-rigidity} and \cite[Theorem 1.9]{HW2} repeatedly,
  we conclude the following:
\begin{enumerate}[\rm(i)]
\item\label{item:refl} $\otimes^r_RM$ is reflexive for all $r=1, \dots, n$. 
\item\label{item:deadTor} $\Tor^R_i(M,\otimes^{r-1}_RM) = 0$ for all
  $i\geq 1$ and all $r=2, \dots, n$.  
\item\label{item:fpd} $\pd(M)<\infty$.
\end{enumerate}
It follows from (i) that $\depth(\otimes_R^rM)\geq 2$ for all $r=1,
\dots, n$; see (\ref{Serre}).  Also, (ii) implies the depth formula: 
\[
\depth(M) + \depth(\otimes^{r-1}_RM) = d + \depth(\otimes^r_RM)\,,
\]
for all $r=2, \dots, n$.  One checks by induction on $r$ that 
\[
r \cdot \depth(M) = (r-1 ) \cdot d + \depth(\otimes_R^rM)\,,
\]
for  $r=2, \dots, n$. Setting $r = n$, and using the inequalities $n\ge d-1$ and
$\depth(\otimes_R^nM)\geq 2$, we obtain:
\[
n \cdot \depth(M) \ge (n-1) \cdot d+2 = n \cdot (d-1) + n-d+2 \ge n \cdot (d-1) + 1.
\]
Therefore $\depth(M)\geq d$, that is, $M$ is maximal Cohen-Macaulay.  Now \eqref{item:fpd} and the Auslander-Buchsbaum formula \cite[Theorem 3.7]{AuBu} imply that $M$ is free.
 \end{proof}
 
A consequence of Theorems~\ref{prop:HW1-main} and \ref{prop:second-rigidity} is the following result  \cite[Theorem 1.9]{HW2}, observed by Huneke and Wiegand in their  1997 paper:

\begin{proposition}[\cite{HW2}]\label{prop:HW2-fin-proj-dim}
Let $M$ and $N$ be finitely generated modules over a hypersurface $R$, and assume that $\Tor_i^R(M,N) = 0$ for $i\gg0$.  Then at least one of the modules has finite projective dimension.
\end{proposition}

At about the same time Miller \cite{Mi} obtained the same result independently, by an elegant, direct argument.  As Miller observed in \cite{Mi}, one can turn things around and easily deduce 
Theorem~\ref{prop:HW1-main} from Proposition~\ref{prop:HW2-fin-proj-dim} and the vanishing result Theorem~\ref{prop:second-rigidity}.

\section*{Acknowledgments}
We would like to thank the referee for a careful reading of the paper.

\bibliographystyle{plain}

\enlargethispage{\baselineskip}

\end{document}